\documentclass{article}
\usepackage{amsmath,amssymb,amsthm,stmaryrd,makeidx}
\usepackage[all]{xy}

\DeclareMathOperator{\der}{Der}
\DeclareMathOperator{\enm}{End}

\DeclareMathOperator{\id}{id}

\DeclareMathOperator{\im}{im}

\DeclareMathOperator{\ann}{Ann}
\DeclareMathOperator{\hmm}{Hom}
\DeclareMathOperator{\mor}{Mor}

\DeclareMathOperator{\ph}{Ph}

\DeclareMathOperator{\pts}{Pts}

\newcommand{\cat}[1]{\mathbf{#1}}

\newcommand{\sets}{\mathbf{Sets}}

\newcommand{\grmcat}[1]

\newcommand{\gl}{\mathsf{GL}}

\input xypic

\newtheorem{proposition}{Proposition}
\newtheorem{theorem}{Theorem}
\newtheorem{lemma}{Lemma}
\newtheorem{corollary}{Corollary}
\newtheorem{definition}{Definition}

\newtheorem{example}{Example}

\makeindex

\begin{document}
\author{Arvid Siqveland}
\title{Riemannian Geometry on Associative Varieties}

\maketitle

\begin{abstract} We prove that the classical algebraic varieties over algebraically closed fields can  be defined over arbitrary fields $k.$ Then we prove that for associative algebras $A$, there exist local representing objects $A_M$ for simple modules $M.$ Replacing the localization in maximal ideals in the commutative situation with the local representations in simple modules in the associative, we define an associative generalization of varieties. Now we realize that replacing $\mathbb R[x_1,\dots,x_n]=\mathbb R[n]$ with $C^\infty(\mathbb R^n),$ we can do differential geometry for associative $\mathbb R[n]$-algebras. This says that we can define a Riemannian geometry on associative varieties. This gives us the definition of connections and algebraic geodesic curves, introducing real geometry into associative algebras.
\end{abstract}

\section{Prologue}
We start with a new interpretation of our observable universe. The Cartesian view considers the centre of the earth in the origin in $\mathbb R^3,$ such that a point is identified with a triple of coordinates. Take a new view of this universe; identify a point in space as a triple of coordinates relative to another point, another triple of coordinates. That is, the universe consists of pairs $(o,p)$ of points called observer and observed where $o,p\in\mathbb R^3,o\neq p.$ Everything is relative: Our model of the universe is invariant under the linear geometric group, that is $U=(\mathbb R^6\setminus\Delta)/\gl(3).$ Assume that there is a yet not understood function $l:U\times U\rightarrow U$ and assume there is a function $v:U\rightarrow\mathbb R$ such that $v(l(P,Q))$ is interpreted as the maximal speed from $P$ to $Q.$ With this view of the universe, time can be defined as a Riemannian metric on $U\times\{l\},$ and we should define $l$ (ad hoc) by claiming all the physical laws to be true.

\section{Classical Algebraic Geometry}
As a reference, we use Harthorne's book \cite{HH77}.
Let $\Bbbk$ be an algebraically closed field, let $f\in\Bbbk[x_1,\dots,x_n]=:\Bbbk[n].$ We put $Z(f)=\{p\in\Bbbk^n|f(p)=0\}$ and let this be a basic closed set for a topology on $\Bbbk^n$. That is $$D(f)=\{p\in\Bbbk^n|f(p)\neq 0\}$$ is a basic open set for the Zariski topology. Hilbert's nullstellensatz says that when $A$ is a finitely generated $k$-algbra, $k$ any field, and $\mathfrak m\subset A$ a maximal ideal, then, in the diagram $$\xymatrix{k\ar[r]\ar[dr]&A\ar[d]^p\\&A/\mathfrak m}$$ $A/\mathfrak m$ is an algebraic extension of $k.$ Thus, when $\Bbbk$ is algebraically closed, the set of maximal ideals in $\Bbbk[n]$ is in one-to-one correspondence with the set $\Bbbk^n,$ where  the maximal ideal $\mathfrak m$ corresponds to $(p(x_1),\dots,p(x_n)),$ and $(p_1,\dots,p_n)$ to the maximal ideal $(x_1-p_1,\dots,x_n-p_n).$

\begin{definition} A topological space $X$ is irreducible, if whenever $X=Z_1\cup Z_2$
with $Z_i\subseteq X$ closed subsets, $i=1,2,$ then $Z_1=X$ or $Z_2=X.$
\end{definition} 

When $X$ is Hausdorff, $X$ is reducible. We can conclude that irreducibility is an algebraic concept. This is more visible when we recall that the algebraic set $$Z(\mathfrak a)=\{p\in\Bbbk^n|f(p)=0\ \forall f\in\mathfrak a\},$$ where $\mathfrak a\subseteq\Bbbk[n]$ is an ideal in the noetherian ring $\Bbbk[n],$ is irreducible if and only if $\mathfrak a$ is prime.

\begin{definition} An Affine Algebraic Variety is an irreducible algebraic subset  $X\subseteq\Bbbk^n.$ Any open subset of an affine algebraic variety is called a quasi-affine algebraic variety.
\end{definition}

An affine (algebraic) variety is then $Z(\mathfrak p)\subseteq\Bbbk^n$ for a prime ideal $\mathfrak p\subseteq\Bbbk^n,$ and as every open subset of an irreducible set is dense and so irreducible, a quasi-affine (algebraic) variety is an open subset $U\subseteq Z(\mathfrak p).$ 

Let $X=Z(\mathfrak a)\subseteq\Bbbk^n$ be an algebraic variety, that is, $\mathfrak a$ is prime, and $A/\mathfrak a$ is a domain. We define the sheaf of regular functions on $X.$ For each open $U\subseteq X,$ put $O_X(U)=\prod_{p\in U}(A/\mathfrak a)_{\mathfrak m_p}.$ We have the natural morphism $$\rho(U):A/\mathfrak a\rightarrow O_X(U)$$ and we define the subset $S(U)\subseteq O_X(U)$ by  $S(U)=\{s^{-1}|s=\rho(U)(a)\text{ unit}\}.$ Then we let $\mathcal O_X(U)$ be the subalgebra of $O_X(U)$ generated by $\im\rho(U)$ and $S(U).$  It follows that $\mathcal O_X$ is a sheaf by the universal property of products. From the diagram $$\xymatrix{\Bbbk[s,s^{-1}]\ar[dr]\ar[d]&\\\mathcal O_X(U)\ar[r]&A_{\mathfrak m_p}}$$ we see that we can lift to any open cover of basic opens.

We also notice that because $A/\mathfrak a$ is a domain, $\rho(U)$ is injective, so that $\mathcal O_X(D(f))=(A/\mathfrak a)_f,$ and in particular $\mathcal O(X)=A/\mathfrak a.$ 

\begin{definition} Let $\mathfrak a\subseteq\Bbbk[n]$ be any ideal, and let $X=Z(\mathfrak a)$ be the corresponding algebraic set. Then define the sheaf $\mathcal O_X$ of regular functions on $X$ as above. 
\end{definition}

\begin{lemma} For each open $U\subseteq Z(\mathfrak a)$ the homomorphism $$\rho=\rho(U):A/\mathfrak a\rightarrow\prod_{p\in U}(A/\mathfrak a)_{\mathfrak m_p}$$ is injective.
\end{lemma}

\begin{proof} $\rho(a)=0\Rightarrow\forall p\in U\ \exists h\notin\mathfrak m_p$ s.t. $ha=0\Rightarrow\forall\mathfrak m_p,\ann(a)\nsubseteq\mathfrak m_p\Rightarrow\ann(a)=A/\mathfrak a\Rightarrow 1\in\ann(a)\Rightarrow 1\cdot a=a=0.$
\end{proof}

\begin{corollary} For any $\mathfrak a\subseteq\Bbbk[n],\ X=Z(\mathfrak a),\ \mathcal O_X(X)=A/\mathfrak a.$
\end{corollary}

A morphism $f:X=Z(\mathfrak p)\rightarrow Z(\mathfrak q)=Y$ between affine algebraic varieties, is a $\Bbbk$-algebra homomorphism $\phi:A(Y)=\Bbbk[n]/\mathfrak q\rightarrow\Bbbk[n]/\mathfrak p=A(X).$

Because of the diagram $$\xymatrix{\Bbbk\ar[r]\ar[drr]_\id&A(Y)\ar[r]^\phi&A(X)\ar[d]\\&&\Bbbk}$$ it follows that maximal ideals are sent to maximal ideals. Note that classically, one says that a morphism $\phi:(X,\mathcal O_X)\rightarrow (Y,\mathcal O_Y)$ consists of a continuous map $\phi:X\rightarrow Y$ such that for every $f\in\mathcal O_Y(U)$ we have that $f\circ\phi:\phi^{-1}(U)\rightarrow\Bbbk$ is in $\mathcal O_X(\phi^{-1}(U)),$ but that is equivalent to the above definition.

Finally, the definition of an algebraic variety over $\Bbbk$ can be made global. We define a ringed topological space as a pair $(X,\mathcal O_X)$ with $X$ a topological space and $\mathcal O_X$ a sheaf of rings $\mathcal O_X:\cat{Top}(X)\rightarrow\cat{Rings}.$ A morphism of ringed spaces is a pair $(f,f^\#):(X,\mathcal O_X)\rightarrow (Y,\mathcal O_Y)$ where $f:X\rightarrow Y$ is continuous and $f^\#:f_{\ast}\mathcal O_Y\rightarrow\mathcal O_X$ is a morphism of functors.

\begin{definition} An abstract variety over $\Bbbk$ is a ringed topological space $(X,\mathcal O_X)$ such that $X$ has an open covering $X=\cup_{i\in I}U_i$ with each $(U_i,\mathcal O_X|_{U_i})\simeq (V,\mathcal O_V)$ an affine variety.
\end{definition}

This last definition is the frame for projective spaces and moduli varieties in general.

\section{Categorical Algebraic Geometry}
Let $k$ be any field and $A$ a finitely generated $k$-algebra. Then the set of $k$-algebra homomorphisms $\hmm_k(A,k)$ is in bijective correspondence with the set of maximal ideals $\mathfrak m$ in $A$ which are such that $A/\mathfrak m\simeq k.$ This is  because every homomorphism $p:A\rightarrow k$ commutes in the diagram $$\xymatrix{k\ar@{^{(}->}[r]\ar[dr]_\id&A\ar@{->>}[d]^p\\&k}$$ so that the correspondence is given by sending $p$ to $\mathfrak m_p=\ker(p).$ We define the space of $k$-points in $A$ as $\pts_k A=\hmm_k(A,k),$ and for $f\in A$ we denote $p(f)=f(p).$ Then for $f\in A$ we let the basic open sets be given by $$D(f)=\{p\in\pts_k  A|f(p)\neq 0\}\subseteq\pts_k(A),$$ thereby defining a topology on $\pts_k A.$ 

Given a $k$-algebra homomorphism $\phi:B\rightarrow A.$ Then for any $p\in\pts_k A,\ \phi$ fits in the diagram $$\xymatrix{k \ar@/_/[drr]_\id\ar@{^{(}->}[r]&B\ar[r]^\phi\ar[dr]&A\ar[d]^p\\&&k}$$ giving a map $\phi:\pts_k A\rightarrow\pts_k B$ such that for all $f\in B,$ 
$$\phi^{-1}(D(f))=\{p\in\pts_k A|f(\phi(p))\neq 0\}=D(\phi(f)),$$ that is, $\phi$ is continuous.

We define a sheaf of $k$-algebras on $X_k=\pts_k A.$ For each open $U\subseteq X_k$ we consider the natural morphism $\rho(U):A\rightarrow\prod_{p\in U}A_{\mathfrak m_p},$ and let $$S(U)=\{s^{-1}\in\prod_{p\in U}A_{\mathfrak m_p}|s=\rho(U)(a)\text{ is a unit}\}.$$ We put $\mathcal O_{X_k}(U)=\im\rho(U)\langle S(U)\rangle,$ that is the subalgebra generated by $\im\rho(U)$ and $S(U).$ Then $\mathcal O_{X_k}$ is a sheaf of $k$-algebras. Also, given a homomorphism $\phi:B\rightarrow A,$ we have the diagram 
$$\xymatrix{B\ar[r]^\phi\ar[d]&A\ar[d]\\\prod_{p\in Y_k}B_{\mathfrak m_{\phi(p)}}\ar[r]&\prod_{p\in X_k}A_{\mathfrak m_p},}$$ inducing a morphism $\phi^\ast:X_k\rightarrow Y_k$ with $Y_k=\pts_k B.$

Let $\max (A)$ be the set of maximal ideals in the $k$-algebra $A,$ let $\rho:A\rightarrow A_k:=\Gamma(\mathcal O_X):=\mathcal O_{\pts_k A}(\pts_k A)$ be the natural homomorphism.

\begin{lemma}
\begin{itemize}
\item[(i)] $\max(A_k)=\pts_k A_k.$
\item[(ii)] $(A_k)_{m_p}\simeq A_{\mathfrak m_{\rho(p)}}$
\item[(iii)] $\Gamma(\mathcal O_{X_k})=\mathcal O_{X_k}(X_k)\simeq A_k.$
\end{itemize}
\end{lemma}

\begin{proof} (i): Assume $\mathfrak m\subset A_k$ maximal, but $A_k/\mathfrak m\ncong k.$ Then there exists an $f\in\mathfrak m$ s.t. $f\notin\mathfrak m_p$ for all $\mathfrak m_p, p\in\pts A_k.$ Then $f=\rho(f)$ is a unit, which contradicts the fact that $\mathfrak m$ is maximal. (ii) follows by the universal property of localization, and (i) and (ii) implies (iii).
\end{proof}

\begin{definition} For any field $k,$ an affine $k$-variety is $(\pts_k A,\mathcal O_{\pts_k A})$ for a $k$-algebra $A.$ A morphism  $\pts_k A\rightarrow \pts_k B$ is a $k$-algebra homomorphism $\phi:B\rightarrow A.$ An abstract $k$-variety is a ringed space $(X,\mathcal O_X)$ covered by open affine $k$-varieties.
\end{definition}

\section{Categorical Differential Geometry}

The polynomial ring in $n$ variables over the real numbers is denoted \newline $\mathbb R[x_1,\dots,x_n]=\mathbb R[n].$ The $\mathbb R$-algebra of infinitely differentiable functions $f:\mathbb R^n\rightarrow\mathbb R$ is denoted $C^\infty(\mathbb R^n).$ We observe that for $C=\mathbb R[n]$ or $C=C^\infty(\mathbb R^n)$ the map $$\Phi:\mathbb R^n\rightarrow\pts_{\mathbb R}(C)$$ sending a point $p\in\mathbb R^n$ to the $\mathbb R$-algebra homomorphism $\phi:C\rightarrow\mathbb R,\ \phi(f)=f(p),$ is bijective, with inverse sending $p:C\rightarrow\mathbb R$ to $(p(x_i),\dots,p(x_n)).$

We recall from \cite{Lee18} that a smooth manifold $X$ is covered by an indexed smooth cover of opens, $X=\cup_{i\in I}U_i,$ where for each $i\in I,U_i\cong\mathbb R^n,$ i.e. $U_i$ is diffeomorphic to $\mathbb R^n.$ Define, for each open $U\subseteq X,$  $\mathcal O_X(U):=C^\infty(U).$ We obtain a sheaf with smooth transition morphisms. This definition is equivalent to the usual definition, because every open subset $U\subseteq\mathbb R^n$ is diffeormorphic to $\mathbb R^n.$ See that $$\tan(\frac{\pi}{e-d}x+\frac{\pi(d+e)}{2(d-e)}):\langle d,e\rangle\rightarrow\mathbb R$$ is a diffeormorphism, and lift to the product topology. Also notice at an open interval $\langle a,b\rangle \subseteq\mathbb R$ can be identified with $\langle a,b\rangle=D(f)=\{x\in\mathbb R|f(x)\neq 0\}$ for a smooth (bump) function $f.$

\begin{definition} An Algebraic Smooth $n$-variety is a topological space $X$ with a sheaf of $\mathbb R$-algebras $\mathcal O_X$ and a covering $X=\cup_{i\in I}U_i$ of open subsets such that for each $i,$ $(U_i,\mathcal O_X|_{U_i})\simeq (\mathbb R^n,\mathcal O^A_{\mathbb R^n})$ where $A=\mathbb R[n].$ The topology has basis $\{D(f)|f\in A\}.$
\end{definition}

\begin{definition} A Smooth Categorical $n$-manifold is a topological space $M$ with a sheaf of  $\mathbb R$-algebras $\mathcal O_X$ and a covering $X=\cup_{i\in I}U_i$ of open subsets such that for each $i,$ $(U_i,\mathcal O_X|_{U_i})\simeq (\mathbb R^n,\mathcal O^A_{\mathbb R^n})$ where $A=C^\infty(\mathbb R^n).$
\end{definition}

These two definitions are identical, so that any geometrical concept on smooth manifolds can be transformed directly to smooth $n$-varieties because their sets of $\mathbb R$-points are in bijective correspondence.

We end this section with the Strong Withney Embedding Theorem: Every smooth $n$-manifold is diffeomorphic to a submanifold $S\subseteq\mathbb R^{2n}.$ This says that the geometry is in the embedding.

\section{Local Representability of Associative Algebras}
For any field $k,$ let the category of pointed $k$-algebras $\cat{Alg}_k^\ast$ have as objects $k$-algebras with a fixed point $p^\ast:A\rightarrow k.$ The morphisms are commutative diagrams $$\xymatrix{A\ar[r]^\phi\ar[d]_{p^\ast}&B\ar[dl]^{q^\ast}\\k}.$$ Let $\cat L\subseteq\cat{Alg}_k^\ast\subseteq\cat{Alg}_k$ be a subcategory. For each  $p:A\rightarrow k\in\cat{Alg}_k^\ast$ we have the functor $$\hmm_k^p(A,-):\cat L\rightarrow\sets.$$

\begin{definition} If there exists a subcategory $\cat L\subseteq\cat{Alg}_k^\ast$ such that $\hmm_k^p(A,-)$ is represented by $A_p$ for all $(A,p),$ $\cat L$ is called a localizing subcategory, and $A_p$ is called locally representing in $p.$
\end{definition}

\begin{example} The set of local (commutative) $k$-algebras is a localizing category in the category of commutative algebras.
\end{example}

\begin{lemma} Let $f:A\rightarrow B$ be a homomorphism of associative rings where $D\subseteq B$ is a division ring. Let $S(f)=f^{-1}(D\setminus\{0\}).$ Then there exists an associative ring $A_{S(f)}$ together with  homomorphisms $g:A\rightarrow A_{S(f)}$ and $i:A_{S(f)}\rightarrow B$ such that $g(s)$ is a unit whenever $s\in S(f),$  and such that $i\circ g=f.$ If $C$ is any associative ring with homomorphisms $h,j$ as above (or in the diagram below), there exist a unique homomorphism $\psi:A_{S(f)}\rightarrow C$ such that $\psi\circ g=h.$
\end{lemma}

\begin{proof} Let $A_{S(f)}\subseteq B$ be the subring generated by $\im f$ together with the set $\{f(s)^{-1}|s\in S(f)\}.$ Consider the diagram $$\xymatrix{A_{S(f)}\ar@/^/[rrd]^i&&\\&A\ar[ul]^g\ar[dl]_h\ar[r]^f&B\supseteq D\\C\ar@/_/[rru]_j&&}$$ where we notice that $i$ is injective. Every element in $A_{S(f)}$ is a polynomial $P(\{f(a_i)\},\{f(s_i)^{-1}\}$ in a finite number of variables, and we have  that $$1=j(h(s)h(s)^{-1})=j(h(s))j(h(s)^{-1})\Rightarrow j(h(s)^{-1})=j(h(s))^{-1}=f(s)^{-1}.$$
It follows that $j(P(\{h(a_i)\},\{h(s_i)^{-1}\})=P(\{f(a_i)\},\{f(s_i)^{-1}\})$ so that 
the unique homomorphism $$\psi:A_{S(f)} \rightarrow C$$ commuting in the above diagram, defined by  $\psi(f(a))=h(a),\psi(f(s)^{-1})=h(s)^{-1},$ is well defined.
\end{proof}

\begin{theorem} There exists a localizing category $\cat L\subseteq\cat{Rings^\ast}\subseteq\cat{Rings}$ such that for each pair $\eta:A\rightarrow\enm_{\mathbb Z}(M)$ of an associative ring $A$ and a simple right module $M,$ there exists a local representing object $A_M.$
\end{theorem}

Let $$\eta^A_M=\oplus_{i=1}^r\eta^A_i:A\rightarrow\enm_{\mathbb Z}(\oplus_{i=1}^rM_i)$$ be a set of $r>0$ simple right $A$-modules and put $M=\oplus_{i=1}^rM_i.$ As any $A$-linear homomorphism is also $\mathbb Z$-linear, we have a canonical ring homomorphism $$\gamma:D_M=\oplus_{i=1}^r\enm_A(M_i)\rightarrow (\hmm_{\mathbb Z}(M_i,M_j))=\enm_\mathbb Z(M)=E_M.$$ This homomorphism is nonzero because it maps the identity to the identity, and so injective because each component $\enm_A(M_i)$ is a division ring. Denote by $D_M^\ast$  the set of units in $D_M,$ that is $D_M^\ast=\{s\in D_M|\gamma_i(s)\neq 0, 1\leq i\leq r\}.$

\begin{definition} The local function ring of $A$ in $M$ is the subring $A_M\subseteq E_M$ of $E_M$ generated over $\im\eta^A_M$ by the subset $\{\eta^A_M(s)^{-1}|\eta^A_M(s)\in D_M\setminus(0)\subset E_M\}.$
\end{definition}

By definition, we find that $A_M=\prod_{i=1}^nA_{M_i}$ is the categorical product of $A_{M_i}.$ 

\begin{lemma}
Let $U$ be any set of simple $A$-modules. Then $$A_U=\underset{\underset{N\subseteq U}\leftarrow}\lim A_N$$ is the categorical product over $U,$ where the inverse limit is taken over finite sets $N$ of simple modules in $U.$ 
\end{lemma}

\begin{proof} Follows by the universal property of inverse (projective) limits.
\end{proof}

\section{Associative Algebraic Varieties}
We apply the theory previously defined to the one-dimensional simple modules only. A possible reference is the book $\cite{S23}.$
Let $k$ be any field, $A$ any associative, not necessarily commutative, $k$-algebra. Then as before, $\pts_k(A)=\hmm_k(A,k)$ with a basis for the topology given by the subsets $D(f)=\{p\in\pts_k(A)|f(p)=p(f)\neq 0\}$ for $f\in A.$  We define a sheaf on $X=\pts_k A$ by for each open $U\subseteq X$ considering the natural morphism $$\rho(U):A\rightarrow\prod_{p\in U}A_p.$$ This is a sheaf by the universal property of products.

\begin{definition}
\begin{itemize}
\item[(i)] An associative affine variety is a pair $(\pts_k A,\mathcal O_{\pts_k A}).$ A morphism $\pts_k A\rightarrow\pts_k B$ is a morphism of ringed spaces induced by a $k$-algebra homomorphism $B\rightarrow A.$
\item[(ii)] An associative algebraic variety is a ringed topological space which is covered by open affine varieties.
\item[(iii)] A morphism of associative varieties is a morphism of ringed spaces which is locally induced by affine morphisms.
\end{itemize}
\end{definition}

\section{Vector Bundles on Associative Varieties}

We start with the definition of vector bundles on a smooth manifold $M.$ Following Lee \cite{Lee18}, Vector Bundle Chart Lemma, Lemma A.32, p. 382, suppose given for each $p\in M$ a real vector space $E_p$ of some fixed dimension $k.$ Let $E=\coprod_{p\in M}E_p,$ and let $\pi:E\rightarrow M$ be the map that takes each element of $E_p$ to the point $p.$ Suppose furthermore that we are given

\begin{itemize}
\item[(i)] an indexed open cover $\{U_\alpha\}_{\alpha\in A}$ of $M$;
\item[(ii)] for each $\alpha\in A,$ a bijective map $\Phi_\alpha:\pi^{-1}(U_\alpha)\rightarrow U_\alpha\times\mathbb R^k$ whose restriction to each $E_p$ is a linear isomorphism from $E_p$ to $\{p\}\times\mathbb R^k\cong\mathbb R^k$;
\item[(iii)] For each $\alpha,\beta\in A$ such that $U_\alpha\cap U_\beta\neq\emptyset,$ a smooth map $\tau_{\alpha\beta}:U_\alpha\cap U_\beta\rightarrow\operatorname{GL}(k,\mathbb R)$ such that the composite map $\Phi_\alpha\circ\Phi_\beta^{-1}$ from $(U_\alpha\cap U_\beta)\times\mathbb R^k$ to itself has the form $$\Phi_\alpha\circ\Phi_\beta(p,v)=(p,\tau_{\alpha\beta}(p)v).$$
\end{itemize}

Then $E$ has a unique smooth manifold structure making it a smooth vector bundle of rank $k$ over $M,$ with $\pi$ as projection and the maps $\Phi_\alpha$ as smooth local trivializations.

Notice in particular item (iii) above. It says that the induced morphisms $\mathbb R^k\rightarrow\mathbb R^k$ is linear, corresponding to an isomorphism of algebras  $\tau:\mathbb R[n]\rightarrow\mathbb R[n].$ 

\begin{lemma} For any field $k,$ a homogeneous homomorphism $f:k[n]\rightarrow k[n]$ defines a linear homomorphism $l(f):k^n\rightarrow k^n.$ A linear map $l:k^n\rightarrow k^n$ defines an algebra homomorphism $f(l):k[n]\rightarrow k[n].$ Also $f=f(l(f))$ and $l=l(f(l)).$ We have that $f$ is an isomorphism if and only if $l(f)$ is an isomorphism.
\end{lemma}

\begin{proof} This follows by applying  the functor $\operatorname{Sym}_k$ to $V.$
\end{proof}

When $M$ is a differentiable manifold, a bundle appears as in the Vector Bundle Chart Lemma. This means that there can be global relations in the pointwise $k$-dimensional vector spaces, as will be clear by the below definition. We generalize to the following.

\begin{definition}\label{bundledef}
\begin{itemize}
\item[(i)] Let $M$ be a categorical differentiable $n$-manifold. A categorical differentiable $n+k$-manifold $E$ is a vector bundle of rank $k$ over $M$ if there is a morphism $\pi:E\rightarrow M$ such that for each open chart $U\subseteq M,$ there is an open chart $V\subseteq\pi^{-1}(U)\subseteq E$ such that the  morphism $$\pi(V):(V,\mathcal O_E|_V)\rightarrow (U,\mathcal O_M|_U)$$ is induced by $$f:C^\infty(\mathbb R^n)\rightarrow (C^\infty(\mathbb R^n)\otimes_{\mathbb R}\mathbb R[k+m])/I=C^\infty(\mathbb R^n)[k+m]/I,$$ where $I$ is an ideal  such that for all points $x:C^\infty(\mathbb R^n)\rightarrow\mathbb R,$ there exists a commutative diagram 
$$\xymatrix{C^\infty(\mathbb R^n)\ar[r]^-f\ar[d]_x&(C^\infty(\mathbb R^n)\otimes_{\mathbb R}\mathbb R[k+m])/I\ar[d]^{x\otimes\id}\\\mathbb R\ar[r]&\mathbb R[k].}$$

\item[(ii)] Let $X$ be a smooth algebraic $n$-variety (i.e. covered by $\pts_k\mathbb R[n]$) . A $k$-bundle $E$ on $X$  is a vector bundle of rank $k$ over $M$ if there is a morphism $\pi:E\rightarrow M$ such that for each open chart $\pts_{\mathbb R} A=U\subseteq M,$ there is an open chart $\pts_{\mathbb R} ((A\otimes_{\mathbb R}\mathbb R\langle k+m\rangle)/I)=V\subseteq\pi^{-1}(U)\subseteq  E$ such that the  morphism $$\pi(V):(V,\mathcal O_E|_V)\rightarrow (U,\mathcal O_M|_U)$$ is induced by $$f:A\rightarrow (A\otimes_{\mathbb R} \mathbb R\langle k+m\rangle)/I$$ where $I$ is an ideal such that for all points $x:A\rightarrow\mathbb R,$ there exists a commutative diagram 
$$\xymatrix{A\ar[r]^-f\ar[d]_x&(A\otimes_{\mathbb R} \mathbb R\langle k+m\rangle)/I\ar[d]^{x\otimes\id}\\\mathbb R\ar[r]&\mathbb R\langle k\rangle.}$$
\end{itemize}
\end{definition}

Observe that in point (ii) in Definition \ref{bundledef}, every point $x\in X$ is assigned the vector space $\pts_{\mathbb R}\mathbb R\langle k \rangle\simeq\mathbb R^k.$

\section{Tangent Spaces on Associative Varieties}

We start the definition of the Phase Space: Let $k$ be a field and let $A$ be an associative $k$-algebra.

\begin{definition} The category $A\text{-}\cat{Alg}_k$ has as objects the $k$-algebra homomorphisms $A\overset\phi\rightarrow B$ from $A$ to a $k$-algebra $B.$ A morphism in this category, is a commutative diagram $$\xymatrix{B\ar[rr]^f&&C\\&A.\ar[ul]^{\phi_1}\ar[ur]_{\phi_2}&}$$
\end{definition}

Consider the functor $$\der_k(A,-):A\text{-}\cat{Alg}_k\rightarrow\sets$$ where $\der_k(A,B)$ is the set of $k$-derivations from $A$ to $B.$ A derivation $d:A\rightarrow B$ is a $k$-linear map which satisfies  $d(ab)=ad(b)+d(a)b.$

Given a morphism of $A$-algebras over $k,$ $f:B\rightarrow C$ and a derivation $d:A\rightarrow B.$ Then $f(d(ab))=f((da)b+a(db))=f(da)b+af(db)$ proving that $\der_k(A,-)$ is a functor.

For the $k$-algebra $A,$  consider the set $dA=\{da|a\in A\},$ and form the quotient of the polynomal algebra $A\langle dA\rangle/D$ where $D$ is the two-sided ideal $D=\langle\{d(ab)-a(db)-(da)b|a,b\in A\}\rangle.$ The map $d:A\rightarrow A\langle dA\rangle/D$ given by $d(a)=da$ is a derivation by definition.

\begin{definition} The associative $A$-algebra over $k,$ $\ph(A)=A\langle dA\rangle/D$ is called the Phase Space of  $A.$ It comes with a $k$-derivation $d:A\rightarrow\ph(A).$
\end{definition}

\begin{lemma} The functor $\der_k(A,-):A\text{-}\cat{Alg}_k\rightarrow\sets$ is represented by $$d:A\rightarrow\ph(A).$$
\end{lemma}

\begin{proof} $\ph(A)$ comes with a derivation $d:A\rightarrow\ph(A).$ If $\rho:A\rightarrow B$ is any other  $A$-algebra over $k$ with a derivation $\overline d:A\rightarrow B,$ then there is a unique morphism $\phi:\ph(A)\rightarrow B$ such that the following diagram commute:
$$\xymatrix{A\ar[r]^d\ar[dr]_{\overline d}&\ph(A)\ar[d]^{\phi}\\&B,}$$ that is $\phi(a(db))=\rho(a)\overline d(b)$ and $\phi((da)b=\overline d(a)\rho(b).$
\end{proof}

\begin{lemma} Let $\phi:A\rightarrow B$ be a $k$-algebra homomorphism. Then there is a natural $k$-algebra homomorphism $\ph(\phi):\ph(A)\rightarrow\ph(B).$
\end{lemma}

\begin{proof} The morphism $\phi:A\rightarrow B$ defines $B$ as an $A$-algebra over $k.$ Let $d_B:B\rightarrow\ph(B)$ be the derivation. Then $d_B\circ\phi(ab)=\phi(a)d_B(\phi(b))+d_B(\phi(a))\phi(b)$ so that $d_B\circ\phi:A\rightarrow B$ is a derivation. Because $\ph(A)$ is representing, there is a unique morphism $\psi:\ph(A)\rightarrow\ph(B)$ of $A$-algebras over $k$ such that $d_A\circ\psi=d_B.$
\end{proof}

We will generalize the previous to geometric associative $n$-varieties. First of all, a geometric associative variety is a scheme $(X,\mathcal O_X)$ of associative rings covered by open affine subsets isomorphic to (the affine scheme) $\pts_{\mathbb R}S$ where $S$ is an $\mathbb R[n]$-algebra, that is, there is an $\mathbb R[n]$-algebra homomorphism $f:\mathbb R[n]\rightarrow S.$

\begin{lemma}{(Gluing Lemma)} Let $\{X_{\alpha}\}_{\alpha\in I}$  be a family of geometric $n$-varieties parametrized by $I.$ Assume that for each pair $\alpha,\beta\in I$  there exists isomorphic open affine subvarieties $U_{\alpha}\subseteq X_{\alpha}, U_{\beta}\subseteq X_{\beta},$  $t_{\alpha\beta}:U_\alpha\overset\sim\rightarrow U_\beta.$ Then there exists a variety $X$ such that $\cup_{\alpha}X_\alpha=X$ is an open covering such that $(X_\alpha,\mathcal O_X|_{X_\alpha})\simeq (X_\alpha,\mathcal O_{X_\alpha}).$
\end{lemma}

\begin{proof} Define an equivalence relation $\sim$ on $\coprod_{\alpha\in I}X_\alpha$ by $x\sim y$ if there exists $\alpha,\beta\in I$ such that there exists isomorphic open affine subvarieties $x\in U_{\alpha}\subseteq X_{\alpha},y\in U_{\beta}\subseteq X_{\beta}$ and an isomorphism $t_{\alpha\beta}:U_\alpha\rightarrow U_\beta$ such that $t_{\alpha\beta}(x)=y.$ Then the topological space $X=\coprod_{\alpha\in I}X_\alpha/\sim$ with the inherited sheaf of rings is called the gluing of $X_\alpha, \alpha\in U.$
\end{proof}

\begin{definition} Let  $M$ be a categorical differentiable $n$-manifold. For each open chart $U\subseteq M$ put $A(U)=C^\infty(\mathbb R^n)$ and then $O_T(U)=\ph(A(U)).$ Then let $TM$ be the gluing of $(\pts_\mathbb R(\ph(A(U)))$ which comes with a projection $\pi:(TM,\mathcal O_{TM})\rightarrow (M,\mathcal O_M).$  $TM$ is called the tangent manifold of $M.$
\end{definition}

\begin{definition} Let $X$ be an associative $n$-variety. For each open affine $U\subseteq X$ put $O_T(U)=\ph(\mathcal O_X(U)).$ Then let $TX$ be the glueing of the affine schemes $\pts_{\mathbb R} \ph(\mathcal O_X(U)).$ Then $TX$ is an associative variety which comes with a projection $TX\rightarrow X.$ $TX$ is called the associative tangent variety of $X.$
\end{definition}

Recall that our mission is to define a geometry on associative varieties. It is reasonable to believe, from the constructions above, that the category of categorical differentiable manifolds is equivalent to the category of (ordinary) differentiable manifolds, and that the tangent bundle on a categorical manifold corresponds to the ordinary tangent bundle. However, we will not use space and time on this here.

\section{Riemannian Geometry}

Let $X$ be an associative geometric variety.  Now $X$ comes with its tangent space $TX\rightarrow X$ and we use the notation $$T^nX=\times_{X}^nTX$$ for the $n$-tensor (this is $TX\times_{X}\cdots\times_{X}TX, n$ times, the $n$-fold fibred product.)

\begin{definition} Let $\psi:E\rightarrow X$ be a vector bundle on a variety $X.$ A vector field on $X$ in $E$ is a section $s:X\rightarrow E,$ that is $\psi\circ s=\id:X\rightarrow X.$ The set of vector fields on $X,$ in $E,$ is denoted $\mathcal E_X.$
\end{definition}

In differential geometry,  tensors are defined as multilinear maps between vector spaces.
Translated to geometric algebras, a vector bundle is a homomorphism $A\rightarrow E,$ so that a $2$-tensor is a homomorphism $$\psi:A\rightarrow\mor_{\cat{Alg}_{\mathbb R[n]}}(E\otimes_A E,\mathbb R[t]).$$ A $2$-tensor field is then commutative diagram $$\xymatrix{A\ar[r]\ar[dr]&E\otimes_A E\\&A\otimes_{\mathbb R}\mathbb R[t].\ar[u]_p}$$ This says that for each point $x:A\rightarrow\mathbb R$ we get,  by tensoring with with $\mathbb R$ over $A,$ the diagram $$\xymatrix{\mathbb R\ar[r]\ar[dr]&\mathbb R\otimes_A(E\otimes_A E)\\&\mathbb R[t]\ar[u]_{p_x}}$$ which induces an $\mathbb R$- bilinear continuous map $$p_x:\mathbb R^k\times\mathbb R^k=\pts(\mathbb R\otimes_A(E\otimes_A E))\rightarrow\pts(\mathbb R[t])=\mathbb R.$$

\begin{definition} A Riemannian metric on an associative variety $X,$ is a $2$-tensor field $g\in\mathcal T^2_X$ such that for each $p\in X,$ $g_p$ is an inner product on $T_pX.$
\end{definition}

\begin{proposition} Every associative $n$-geometric variety over $\mathbb R$ has a Riemannian geometry.
\end{proposition}

\begin{proof} It suffices to consider the affine situation. Let $A$ be an associative, $n$-geometric algebra for some $n.$ Then its tangent space is $\ph(A),$ and a $2$-tensor field $g\in\mathcal T^2_X$ corresponds to a homomorphism $$g:A\otimes_{\mathbb R}\mathbb R[t]\rightarrow\ph(A)\otimes_A\ph(A).$$ Because $A$ is finitely generated, we have $A\simeq(\mathbb R[x_1,\dots,x_n]\otimes_\mathbb R\mathbb R\langle x_{n+1},\dots,x_m\rangle)/J$ for some ideal $J,$ and so $$\ph(A)\simeq\mathbb R\langle x_1,\dots,x_m,dx_1,\dots,dx_m\rangle/(dJ)$$ for some $m\geq n.$ Then $$\ph(A)\otimes_A\ph(A)=\mathbb R\langle x_1,\dots,x_m,dx_1,\dots,dx_m,dy_1,\dots,dy_m\rangle/(J,dJ\otimes_A dJ).$$ Every $dx_i$ is actually $\frac{d}{dx_i}|_{(x_1,\dots,x_m)},$
so we define $g:\mathbb R[t]\rightarrow\ph(A)\otimes_A\ph(A)$ by $$g(t)=\sum_{i=1}^m (dx_i)(dy_i).$$ This induces the usual Euclidean metric on $\mathbb R^m$ which glues to a Riemannian metric on $X.$  
\end{proof}

In the case of a smooth manifold $M$, with the Euclidean topology, we need the existence of a partition of unity for the proof of the existence of a Riemannian geometry on $M.$ This is avoided in the algebraic case because of the Zariski topology which gives the definition of a not necessarily Hausdorff topology so that partition of unity follows from the irreducibility of a variety.

\section{Epilogue}
Let $U$ be an associative variety such that each $\mathbb R$-point is an equivalence class of (observer,observed), $(o,p),$ under the action of the linear group $\gl_6.$ We can obtain $U$ by blowing up $\mathbb R^6$ noncommutatively in the diagonal $\Delta\subseteq\mathbb R^3\times\mathbb R^3,$ which is adding a noncommutative tangent vector $\vec c.$ Choose a metric on $U,$ and $\vec c$ defines the velocity $v$ from $P=(o_1,p_1)$ to $Q=(o_2,p_2)$ and then time is $t=v/d(P,Q).$ In Laudal's book \cite{Laudal21}, many physical laws is proved to hold.

\end{document}